\documentclass[11pt]{amsart}
\usepackage{amscd,amssymb}
\input xy
\xyoption{all}
\newcommand{\cov}{\mathrm{cov}}
\newcommand{\Cov}{\mathrm{Cov}}

\newcommand{\e}{\varepsilon}
\newcommand{\w}{\omega}

\newcommand{\diam}{\mathrm{diam}}
\newcommand{\Ra}{\Rightarrow}

\newcommand{\IR}{\mathbb R}
\newcommand{\IZ}{\mathbb Z}
\newcommand{\IN}{\mathbb N}
\newcommand{\upa}{\uparrow}
\newcommand{\da}{\downarrow}
\newcommand{\lev}{\mathrm{lev}}
\newcommand{\Lev}{\mathrm{Lev}}
\newcommand{\id}{\mathrm{id}}
\newcommand{\suc}{\mathrm{pred}}
\newcommand{\Deg}{\mathrm{Deg}}

\newcommand{\vx}{\mathrm{vx}}
\newtheorem{theorem}{Theorem}[section]
\newtheorem{proposition}[theorem]{Proposition}
\newtheorem{lemma}[theorem]{Lemma}

\newtheorem{claim}[theorem]{Claim}
\newtheorem{remark}[theorem]{Remark}

\title[Classifying homogeneous ultrametric spaces...]{Classifying homogeneous ultrametric \\ spaces up to coarse equivalence}
\author[T. Banakh and D. Repov\v s]{Taras Banakh and Du\v san Repov\v s}
\address{T.Banakh: Ivan Franko National University, Lviv, Ukraine
\&
 Jan Kochanowski University, Kielce, Poland}
\email{t.o.banakh@gmail.com}
\address{D.Repov\v s: Faculty of Education, and
Faculty of Mathematics and Physics,
University of Ljubljana
\& Institute of Mathematics, Physics and Mechanics, 
Ljubljana, Slovenia}
\email{dusan.repovs@guest.arnes.si}

\begin{document}
\begin{abstract} For every metric space $X$ we introduce two cardinal characteristics $\cov^\flat(X)$ and $\cov^\sharp(X)$ describing the capacity of balls in $X$. We prove that these cardinal characteristics are invariant under coarse equivalence and prove that two ultrametric spaces $X,Y$ are coarsely equivalent if $\cov^\flat(X)=\cov^\sharp(X)=\cov^\flat(Y)=\cov^\sharp(Y)$. This result implies  that an ultrametric space $X$ is coarsely equivalent to an isometrically homogeneous ultrametric space if and only if $\cov^\flat(X)=\cov^\sharp(X)$. Moreover, two isometrically homogeneous ultrametric spaces $X,Y$ are coarsely equivalent if and only if $\cov^\sharp(X)=\cov^\sharp(Y)$ if and only if each of these spaces coarsely embeds into the other space. This means that the coarse structure of an isometrically homogeneous ultrametric space $X$ is completely determined by the value of the cardinal $\cov^\sharp(X)=\cov^\flat(X)$.
\end{abstract}
\maketitle

\section{Introduction and Main Results}

In this paper we present a criterion for recognizing coarsely equivalent ultrametric spaces and apply this to classify isometrically homogeneous ultrametric spaces up to coarse equivalence.
Let us recall that an {\em ultrametric space} is a metric space $(X,d)$ whose metric satisfies the strong triangle inequality: $d(x,z)\le\max\{d(x,y),d(y,z)\}$ for all points $x,y,z\in X$.
A metric space $(X,d)$ is called {\em isometrically homogeneous} if for any points $x,y\in X$ there is an isometric bijection $f:X\to X$ such that $f(x)=y$. A typical example of an isometrically homogeneous metric space is any group $G$ endowed with a left-invariant metric $d$.

We shall be interested in classifying isometrically homogeneous ultrametric spaces up to coarse equivalence. A map $f:X\to Y$ between two metric spaces $X,Y$ is called {\em coarse} if for any $\e\in\IR_+$ there is $\delta\in\IR_+$ such that for any subset $A\subset X$ of diameter $\diam\,A\le\e$ its image $f(A)$ has diameter $\diam\,f(A)\le\delta$. Here by $\IR_+$ we denote the half-line $(0,\infty)$. For a subset $A$ of a metric space $(X,d_X)$ its {\em diameter} is defined as expected: $\diam A=\sup_{x,y\in A}d_X(x,y)$.

A bijective map $f:X\to Y$ between two metric spaces is called a {\em coarse isomorphism} if both maps $f$ and $f^{-1}$ are coarse. In this case the metric spaces $X,Y$ are called {\em coarsely isomorphic}. Two metric spaces $X,Y$ are called {\em coarsely equivalent} if they contain coarsely isomorphic large subspaces $L_X\subset X$ and $L_Y\subset Y$. A subset $L\subset X$ of a metric space $(X,d_X)$ is called {\em large} if $X=\bigcup_{x\in L}B_\e(x)$ where $B_\e(x)=\{y\in X:d_X(x,y)\le \e\}$ stands for the closed $\e$-ball centered at $x$. It follows that each metric space $X$ is coarsely equivalent to any large subset in $X$. For example the real line $\IR$ is coarsely equivalent to the space of integers $\IZ$.

Properties of metric spaces preserved by coarse equivalences are studied in Coarse (or else Asymptotic)  Geometry \cite{BS}--\cite{Roe}.
In this paper we shall classify isometrically homogeneous ultrametric spaces up to coarse equivalence thus extending the classification of separable isometrically homogeneous ultrametric spaces given in  \cite{Zar}. According to \cite{Zar}, each isometrically homogeneous separable ultrametric space is coarsely equivalent to one of three spaces: the singleton $1$, the Cantor macro-cube $2^{<\IN}$ or the Baire macro-space $\w^{<\IN}$.

In this paper we shall prove that the coarse structure of an isometrically homogeneous ultrametric space $X$ is fully determined by the value of two (coinciding) cardinal invariants $\cov^\flat(X)$ and $\cov^\sharp(X)$, which are defined for any metric space $X$ as follows.

For a point $x\in X$ of a metric space $X$ and two numbers $\e,\delta\in\IR_+$ let
$$\cov_\e^\delta(x)=\min\{|C|:C\subset X,\;B_\delta(x)\subset \bigcup_{c\in C}B_\e(x)\}$$ be the smallest number of closed $\e$-balls covering the closed $\delta$-ball centered at $x$.

For a metric space $X$ let
\begin{itemize}
\item $\cov^\sharp(X)$ be the smallest cardinal $\kappa$ for which there is $\e\in\IR_+$ such that for every $\delta\in\IR_+$ we get $\sup_{x\in X}\cov_\e^\delta(x)<\kappa$;
\item $\cov^\flat(X)$ be the largest cardinal $\kappa$ such that for every cardinal $\lambda<\kappa$ and real number $\e\in\IR_+$ there is $\delta\in\IR_+$ such that $\min_{x\in X}\cov_\e^\delta(x)\ge\lambda$.
\end{itemize}
It follows that $\cov^\flat(X)\le\cov^\sharp(X)$ and the cardinals $\cov^\flat(X)$ and $\cov^\sharp(X)$ can be equivalently defined as
$$\cov^\sharp(X)=\min_{\e\in\IR_+}\sup_{\delta\in\IR_+}\big(\sup_{x\in X}\cov_\e^\delta(x)\big)^+$$
 and
 $$\cov^\flat(X)=\min_{\e\in\IR_+}\sup_{\delta\in\IR_+}\big(\min_{x\in X}\cov_\e^\delta(x)\big)^+,$$
where $\kappa^+$ denotes the smallest cardinal which is larger than $\kappa$.
Cardinals are identified with the smallest ordinals of a given cardinality.

The following proposition on coarse invariance of the cardinal characteristics $\cov^\flat$ and $\cov^\sharp$ will be proved in Section~\ref{s3}.

\begin{proposition}\label{p1} If metric space $X$ and $Y$ are coarsely equivalent, then  $$\cov^\flat(X)=\cov^\flat(Y)\mbox{ \ and \ }\cov^\sharp(X)=\cov^\sharp(Y).$$
\end{proposition}

Observe that the inequality $\cov^\sharp(X)\le\w$ means that $X$ has bounded geometry while $\cov^\flat(X)\ge\w$ means that $X$ has no isolated balls (see \cite{Zar} for definitions). By \cite{BZ}, any two ultrametric spaces of bounded geometry and without isolated balls are coarsely equivalent.

The following criterion of coarse equivalence of ultrametric spaces generalizes this fact and is one of the principal results of this paper.

\begin{theorem}\label{main} Let $X,Y$ be two ultrametric spaces.
\begin{enumerate}
\item If $\cov^\sharp(X)\le\cov^\flat(Y)$, then $X$ is coarsely equivalent to a subspace of $Y$.
\item If $\cov^\flat(X)=\cov^\sharp(X)=\cov^\sharp(Y)=\cov^\flat(Y)$, then $X$ and $Y$ are coarsely equivalent.
\end{enumerate}
\end{theorem}

Theorem~\ref{main} will be proved in Section~\ref{s5} after some preparatory work in Section~\ref{s4}.
Now we shall present some applications of this theorem.

The first one is the characterization of ultrametric spaces $X$ with $\cov^\flat(X)=\cov^\sharp(Y)$.

\begin{theorem}\label{iso} An ultrametric space $X$ is coarsely equivalent to an isometrically homogeneous ultrametric space if and only if $\cov^\flat(X)=\cov^\sharp(X)$.
\end{theorem}

\begin{proof} If an ultrametric space $X$ is isometrically homogeneous, then for any points $x,y\in X$ and numbers $\e,\delta\in\IR_+$ we get $\cov_\e^\delta(x)=\cov_\e^\delta(y)$, which implies that $\min_{x\in X}\cov_\e^\delta(x)=\sup_{x\in X}\cov_\e^\delta(y)$ and hence $\cov^\flat(X)=\cov^\sharp(X)$.

If an ultrametric space $X$ is coarsely equivalent to an isometrically homogeneous metric space $Y$, then the invariance of the cardinal characteristics $\cov^\flat$ and $\cov^\sharp$  under coarse equivalences implies that $\cov^\flat(X)=\cov^\flat(Y)=\cov^\sharp(Y)=\cov^\sharp(X)$. This completes the proof of the ``only if''part of the theorem.

To prove the ``if'' part, assume that $X$ is an ultrametric space with $\kappa=\cov^\flat(X)=\cov^\sharp(X)$. The definition of $\kappa=\cov^\flat(X)=\cov^\sharp(X)$ implies that either $\kappa=0$ or $\kappa=1$ or $\kappa$ has countable cofinality or $\kappa$ is a successor cardinal.

If $\kappa=0$, then the space $X$ is empty and hence isometrically homogeneous.

If $\kappa=1$, then the ultrametric space $X$ is bounded and is coarsely equivalent to the singleton (which is an isometrically homogeneous ultrametric space).

If $\kappa$ has countable cofinality or is a successor cardinal, then we can choose a non-decreasing sequence of non-zero cardinals $(\kappa_n)_{n\in\IN}$ such that $\kappa=\sup_{n\in\IN}\kappa_n^+$.
Choose an increasing sequence of groups $\{e\}=G_0\subset G_1\subset G_2\subset\dots$ such that $|G_{n}/G_{n-1}|=\kappa_n$ and on the union $G=\bigcup_{n\in\w}G_n$ consider the left-invariant ultrametric
$$d_G(x,y)=\min\{n\in\w:x^{-1}y\in G_n\}$$turning $G$ into an isometrically homogeneous
 ultrametric space $(G,d_G)$.

Observe that $$\cov^\flat(G,d_G)=\cov^\sharp(G,d_G)=\min_{n\in\IN}\sup_{m\ge n}|G_m/G_n|^+=\sup_{m\in\w}\kappa_m^+=\kappa.$$
Applying Theorem~\ref{main}, we conclude that the ultrametric space $X$ is coarsely equivalent to the isometrically homogeneous ultrametric space $(G,d_G)$.
\end{proof}

Theorem~\ref{main} and Propositions~\ref{p1} imply the following classification of isometrically homogeneous ultrametric spaces.

\begin{theorem} For isometrically homogeneous ultrametric spaces $X,Y$ the following conditions are equivalent:
\begin{enumerate}
\item $X$ and $Y$ are coarsely equivalent;
\item $X$ is coarsely equivalent to a subspace of $Y$ and $Y$ is coarsely equivalent to a subspace of $X$;
\item $\cov^\sharp(X)=\cov^\sharp(Y)$.
\end{enumerate}
\end{theorem}

\begin{proof} The implication $(1)\Ra(2)$ is trivial, the implication $(2)\Ra(3)$ follows from Proposition~\ref{p1} and monotonicity of $\cov^\sharp$ under taking subspaces (see Lemma~\ref{l3.1}), and the final implication $(3)\Ra(1)$ follows from Theorems~\ref{main} and \ref{iso}.
\end{proof}

It is well known \cite{BDHM} that a metric space $X$ is coarsely equivalent (even coarsely isomorphic) to an ultrametric space if and only if $X$ has asymptotic dimension $\mathrm{asdim}(X)=0$. So, in fact all our results concern the coarse classification of metric spaces of asymptotic dimension zero.

In conclusion, we briefly describe the structure of the paper.
In Section~\ref{s2} we characterize the coarse equivalences by means of macro-uniform multivalued maps. Section~\ref{s3} contains the proof of Proposition~\ref{p1}. In Section~\ref{s4} we recall the necessary information about towers and their morphisms and in  Section~\ref{s5} we present a proof of our main result (Theorem~\ref{main}).

\section{Characterizing coarse equivalences}\label{s2}

In this section we shall discuss the definition of coarse equivalence based on the notion of a multi-map.
Such approach to defining coarse equivalences was suggested and exploited in \cite{BZ}.

By a {\em multi-map} $\Phi:X\multimap Y$ between two sets $X,Y$ we
understand any subset $\Phi\subset X\times Y$. For a subset $A\subset X$ by $\Phi(A)=\{y\in Y:\exists a\in
A\mbox{ with }(a,y)\in\Phi\}$ we denote the image of $A$ under the
multi-map $\Phi$. Given a point $x\in X$ we write $\Phi(x)$
instead of $\Phi(\{x\})$.

The inverse $\Phi^{-1}:Y\multimap X$ of the multi-map $\Phi$ is the
multi-map $$\Phi^{-1}=\{(y,x)\in Y\times X: (x,y)\in\Phi\}\subset
Y\times X$$ assigning to each point $y\in Y$ the set $\Phi^{-1}(y)=\{x\in X:y\in\Phi(x)\}$. For two multi-maps $\Phi:X\multimap Y$ and $\Psi:Y\multimap Z$ we
define their composition $\Psi\circ\Phi:X\multimap Z$ as usual:
$$\Psi\circ\Phi=\{(x,z)\in X\times Z:\exists y\in Y\mbox{ such that $(x,y)\in \Phi$ and $(y,z)\in\Psi$}\}.$$


A multi-map $\Phi:X\multimap Y$ between metric spaces $X,Y$ is called {\em coarse} if for any $\e\in\IR_+$ there is $\delta\in\IR_+$ such that for any subset $A\subset X$ of diameter $\diam\, A\le\e$ its image $\Phi(A)$ has diameter $\diam\,\Phi(A)\le\delta$. This is equivalent to saying that for every $\e\in\IR_+$ the oscillation $$\w_\Phi(\e)=\sup\{\diam\,\Phi(A):A\subset X,\;\diam(A)\le\e\}$$is finite. Here $\diam\,A=\sup_{x,y\in A}d_X(x,y)$ is the diameter of a subset $A\subset X$ in a metric space $(X,d_X)$. By definition, $\diam\, \emptyset=0$.

A multi-map $\Phi:X\multimap Y$ between metric spaces is called a {\em coarse embedding} if $\Phi^{-1}(Y)=X$ and both multi-maps $\Phi$ and $\Phi^{-1}$ are coarse. If, in addition, $\Phi(X)=Y$, then the multi-map $\Phi$ is called a {\em coarse equivalence} between $X$ and $Y$.

It is clear that for two coarse embeddings (coarse equivalences) $\Phi:X\multimap Y$ and $\Phi:Y\multimap Z$ their composition $\Psi\circ\Phi:X\multimap Z$ is a coarse embedding (coarse equivalence).

The following characterization of coarse equivalence was proved in Proposition 2.1 of \cite{BZ}.

\begin{proposition}\label{p2.1} For metric spaces $X,Y$ the following conditions are equivalent:
\begin{enumerate}
\item $X$ and $Y$ are coarsely equivalent (i.e., contain coarsely isomorphic large subspaces);
\item there is a coarse equivalence $\Phi:X\multimap Y$;
\item there are coarse maps $f:X\to Y$ and $g:Y\to X$ such that\newline $\sup_{x\in X}d_X(x,g\circ f(x))<\infty$ and $\sup_{y\in Y}d_Y(y,f\circ g(x))<\infty$.
\end{enumerate}
\end{proposition}

\section{Proof of Proposition~\ref{p1}}\label{s3}

Proposition~\ref{p1} follows from two lemmata.

\begin{lemma}\label{l3.1} If a metric space $X$ is coarsely equivalent to a subspace of a metric space $Y$, then $\cov^\sharp(X)\le\cov^\sharp(Y)$.
\end{lemma}

\begin{proof} Proposition~\ref{p2.1} implies that the space $X$, being coarsely equivalent to a subspace of $Y$, admits a coarse embedding $\Phi:X\multimap Y$. By definition of the cardinal $\cov^\sharp(Y)$, there is a number $\e\in\IR_+$ such that for every $\delta\in\IR_+$ we get $\kappa_\delta:=\sup_{y\in Y}\cov_\e^\delta(y)<\cov^\sharp(Y)$.

Taking into account that the multi-map $\Phi^{-1}:Y\multimap X$ is coarse, we conclude that the number $\e'=\w_{\Phi^{-1}}(2\e)$ is finite. The inequality $\cov^\sharp(X)\le\cov^\sharp(Y)$ will follow as soon as we check that $\sup_{x\in X}\cov_{\e'}^{\delta'}(x)<\cov^\sharp(Y)$ for every $\delta'\in\IR_+$. Given any $\delta'\in\IR_+$ consider the finite number $\delta=\w_{\Phi}(2\delta')$ and observe that for every $x\in X$ the ball $B_{\delta'}(x)$ has diameter $\diam\, B_{\delta'}(x)\le2\delta'$, which implies that $\diam\, \Phi(B_\delta(x))\le\w_{\Phi}(2\delta')=\delta$. Then $\Phi(B_{\delta'}(x))\subset B_\delta(y)$ for some point $y\in Y$. Since $\cov_{\e}^\delta(y)\le\kappa_\delta$, there is a subset $C\subset Y$ of cardinality $|C|\le\kappa_\delta$ such that $ B_{\delta}(y)\subset\bigcup_{c\in C}B_\e(c)$. The inclusion $\Phi(B_{\delta'}(x))\subset B_\delta(y)$ and the equality $X=\Phi^{-1}(Y)$ imply that $$B_{\delta'}(x)\subset \Phi^{-1}\big(\Phi(B_{\delta'}(x))\big)\subset \Phi^{-1}(B_\delta(y))\subset\bigcup_{c\in C}\Phi^{-1}(B_\e(c)).$$ For every $c\in C$ the set $\Phi^{-1}(B_\e(c))\subset X$ has diameter $\le\w_{\Phi^{-1}}(2\e)=\e'$ and hence is contained in the closed $\e'$-ball $B_{\e'}(x_c)$ centered at some point $x_c\in X$. Then $B_{\delta'}(x)\subset\bigcup_{c\in C}B_{\e'}(x_c)$, which implies $\cov_{\e'}^{\delta'}(x)\le|C|\le\kappa_\delta$. Therefore, we obtain the desired inequality $\sup_{x\in X}\cov_{\e'}^{\delta'}(x)\le\kappa_\delta<\cov^\sharp(Y)$.
\end{proof}

\begin{lemma} If  metric spaces $X$ and $Y$ are coarsely equivalent, then  $$\cov^\flat(X)=\cov^\flat(Y).$$
\end{lemma}

\begin{proof} By Proposition~\ref{p2.1}, there is a coarse equivalence $\Phi:X\multimap Y$. By symmetry, it suffices to prove that $\cov^\flat(X)\le\cov^\flat(Y)$. This inequality will follow as soon as for every cardinal $\kappa<\cov^\flat(X)$ and every $\e\in\IR_+$ we find $\delta\in\IR_+$ such that $\min_{y\in Y}\cov_{\e}^\delta(y)\ge\kappa$. Given any $\e\in\IR_+$, consider the finite number $\e'=\w_{\Phi^{-1}}(2\e)$ and using the definition of the cardinal $\cov^\flat(X)>\kappa$, find a number $\delta'\in\IR_+$ such that $\min_{x\in X}\cov_{\e'}^{\delta'}(x)\ge\kappa$. We claim that the number $\delta=\w_{\Phi}(2\delta')$ has the required property. Given any point $y\in Y$, we need to check that $\cov_\e^\delta(y)\ge\kappa$. Assuming to the contrary, we could find a set $C\subset Y$ of cardinality $|C|<\kappa$ such that $B_\delta(y)\subset\bigcup_{c\in C}B_\e(c)$. Then for any point $x\in\Phi^{-1}(y)$ we would get $$B_{\delta'}(x)\subset \Phi^{-1}(B_\delta(y))\subset \bigcup_{c\in C}\Phi^{-1}(B_\e(c))\subset \bigcup_{c\in C}B_{\e'}(x_c)$$for any points $x_c\in\Phi^{-1}(c)$, $c\in C$ (see the proof of Lemma~\ref{l3.1}).
This would imply that $\cov_{\e'}^{\delta'}(x)\le|C|<\kappa$, which contradicts the choice of $\delta'$.

So, $\min_{y\in Y}\cov_\e^\delta(y)\ge\kappa$ and hence $\cov^\flat(Y)\ge\cov^\flat(X)$.
\end{proof}

\begin{remark} Easy examples show that the cardinal characteristic $\cov^\flat$ is not monotone with respect to taking subspaces (in contrast to $\cov^\sharp$, which is monotone, according to Lemma~\ref{l3.1}).
\end{remark}

\section{Towers and their morphisms}\label{s4}

Theorem~\ref{main} announced in the introduction will be proved by induction on partially ordered sets called towers. The technique of towers was created in \cite{BZ} for characterization of the Cantor macro-cube. In this section we recall the necessary information on towers.

\subsection{Partially ordered sets} A {\em partially ordered set} is a set $T$ endowed with a reflexive antisymmetric transitive relation $\le$.

A partially ordered set $T$ is called {\em $\upa$-directed}  if for any two points $x,y\in T$ there is a point $z\in T$ such that $z\ge x$ and $z\ge y$.

A subset $C$ of a partially ordered set $T$ is called  {\em cofinal} if for every $x\in T$ there is $y\in C$ such that  $y\ge x$.

By the {\em lower cone} (resp. {\em upper cone}) of a point $x\in T$  we understand  the set ${\downarrow}x=\{y\in T:y\le x\}$ (resp. ${\uparrow}x=\{y\in T:y\ge x\}$). A subset $A\subset T$ will be called a {\em lower} (resp. {\em upper}) {\em set} if ${\downarrow}a\subset A$ (resp. ${\uparrow}a\subset A$) for all $a\in A$.
For two points $x\le y$ of $T$ the intersection $[x,y]={\upa}x\cap {\da}y$ is called the {\em order interval} with end-points $x,y$.

A partially ordered set $T$ is a {\em tree} if $T$ has the smallest element and for each point $x\in T$ the lower cone ${\da}x$ is well-ordered (in the sense that each subset $A\subset{\da}x$ has the smallest element).

\subsection{Defining towers}

A partially ordered set $T$ is called a {\em tower} if
$T$ is $\upa$-directed and for every points $x\le y$ in $T$ the order interval $[x,y]\subset T$ is finite and linearly ordered.

This definition implies that for every point $x$ in a tower $T$ the upper set ${\upa}x$ is linearly ordered and is order isomorphic to a subset of $\w$. Since $T$ is $\upa$-directed, for any points $x,y\in T$ the upper sets ${\upa}x$ and ${\upa}y$ have non-empty intersection and this intersection has the smallest element $x\wedge y=\min({\upa}x\cap{\upa}y)$ (because each order interval in $X$ is finite). Thus any two points $x,y$ in a tower have the smallest upper bound $x\wedge y$.

It follows that for each point $x\in T$ of a tower $T$ the lower cone ${\da}x$ endowed with the reverse partial order is a tree of at most countable height.

\subsection{Levels of  tower} Given two points $x,y\in T$ we write $\lev_T(x)\le\lev_T(y)$ if $$|[x,x\wedge y]|\ge|[y,x\wedge y]|.$$ Also we write $\lev_T(x)=\lev_T(y)$ if  $|[x,x\wedge y]|=|[y,x\wedge y]|$.

The relation $$\{(x,y)\in T\times T:\lev_T(x)=\lev_T(y)\}$$ is an equivalence relation  on $T$ dividing the tower $T$ into equivalence classes called the {\em levels} of $T$. The level containing a point $x\in T$ is denoted by $\lev_T(x)$.
Let $$\Lev(T)=\{\lev_T(x):x\in T\}$$ denote the set of levels of $T$ and let
$$\lev_T:T\to\Lev(T),\;\lev_T:x\mapsto\lev_T(x),$$
stand for the quotient map called the {\em level map}.

The set $\Lev(T)$ of levels of $T$ endowed with the order $\lev_T(x)\le \lev_T(y)$ is a linearly ordered set, order isomorphic to a subset of integers.
For a level $\lambda\in\Lev(T)$ by $\lambda+1$ (resp. $\lambda-1$) we denote the successor (resp. the predecessor) of $\lambda$ in the level set $\Lev(T)$. If $\lambda$ is a maximal (resp. minimal) level of $T$, then we put $\lambda+1=\emptyset$ (resp. $\lambda-1=\emptyset$).

It is clear that each $\upa$-directed subset $S$ of a tower $T$ is a tower with respect to the partial order inherited from $T$. In this case we say that $S$ is a {\em subtower} of $T$. A typical example of a subtower of $T$ is a {\em level subtower}  $$T^L=\{x\in T:\lev_T(x)\in L\},$$ where $L\subset\Lev(T)$ is a cofinal subset of the level set of the tower $T$.

A tower $T$ will be called {\em $\da$-bounded} (resp. {\em $\upa$-bounded}\/) if the level set $\Lev(T)$ has the smallest (resp. largest) element. Otherwise $T$ is called {\em $\da$-unbounded} (resp. {\em $\upa$-unbounded}\/). All towers that will be considered in this paper are assumed to be $\upa$-unbounded and $\da$-bounded.

The level set $\Lev (T)$ of a $\da$-bounded tower can be identified with $\w$, so that zero corresponds to the smallest level of $T$.

\subsection{The boundary of a tower}
By a {\em branch} of a tower $T$ we understand a maximal linearly ordered subset of $T$. The family of all branches of $T$ is denoted by $\partial T$ and is called the {\em boundary} of $T$.
Each branch of a ${\da}$-bounded tower can be identified with its smallest element.
 The boundary $\partial T$ of a ${\da}$-bounded tower carries an ultrametric that can be defined as follows.

Given two branches $x,y\in\partial T$ let
$$\rho(x,y)=\begin{cases}0,&\mbox{if $x=y$,}\\
\lev_T(\min x\cap y),&\mbox{if $x\ne y$.}
\end{cases}
$$
Here we identify the level set $\Lev(T)$ of $T$ with $\w$. It is a standard exercise to check that $\rho$ is a well-defined ultrametric on the boundary $\partial T$ of $T$ turning $\partial T$ into an  ultrametric space.

In the sequel we shall assume that the boundary $\partial T$ of any tower $T$ is endowed with the ultrametric $\rho$.

\subsection{Degrees of points of a tower}

For a point $x\in T$ and a level $\lambda\in\Lev(T)$ let $\suc_\lambda(x)=\lambda\cap{\downarrow}x$ be the set of predecessors of $x$ on the $\lambda$-th level and $\deg_\lambda(x)=|\suc_\lambda(x)|$. For $\lambda=\lev_T(x)-1$,  the set $\suc_{\lambda}(x)$, called the set of {\em parents} of $x$, is denoted by $\suc(x)$. The cardinality $|\suc(x)|$ is called the {\em degree} of $x$ and is denoted by $\deg(x)$. Thus $\deg(x)=\deg_{\lev_T(x)-1}(x)$. It follows that $\deg(x)=0$ if and only if $x$ is a minimal element of $T$.

For levels $\lambda,l\in\Lev(T)$ let
$$\deg_\lambda^l(T)=\min\{\deg_\lambda(x):\lev_T(x)=l\}$$ and$$\Deg_\lambda^l(T)=\sup\{\deg_\lambda(x):\lev_T(x)=l\}.$$

Now let us introduce several notions related to degrees. We define a tower $T$ to be
\begin{itemize}
\item {\em homogeneous} if $\deg_\lambda^l(T)=\Deg_\lambda^l(T)$ for any levels $\lambda\le l$ of $T$;
\item {\em pruned} if $\deg_{\lambda-1}^\lambda(T)>0$ for every non-minimal level $\lambda$ of $T$.
\end{itemize}

It is easy to check that a tower $T$ is pruned if and only if each branch of $T$ meets each level of $T$. In this case the boundary $\partial T$ of a ${\da}$-bounded tower $T$ can be identified with the smallest level of $T$.

There is a direct interdependence between the degrees of points of the tower $T$ and the capacities of the balls in the ultrametric space $\partial T$. For an arbitrary branch $x\in \partial T$ we can see that $\cov_\lambda^l(x)=\deg_\lambda(x\cap \lev_T^{-1}(l))$ for any levels $\lambda\le l\in\Lev(T)=\w$. This implies that $\deg_\lambda^l(T)=\cov_\lambda^l(\partial T)$ and $\Deg_\lambda^l(T)=\Cov_\lambda^l(\partial T)$.


\subsection{Assigning a tower to an ultrametric space} In the preceding section we have assigned to each tower $T$ the ultrametric space $\partial T$. In this section we describe the converse operation assigning to each ultrametric space $X$ a pruned tower $T_X^L$ whose boundary $\partial T_X^L$ is coarsely equivalent to the space $X$.

A closed discrete unbounded subset $L\subset[0,\infty)$ will be called a {\em level set}.
Given an ultrametric space $X$ and a level set $L\subset[0,\infty)$ consider the set $$T^L_X=\{(B_\lambda(x),\lambda):x\in X,\,\lambda\in L\}$$endowed with the partial order $(B_\lambda(x),\lambda)\le (B_l(y),l)$ if $\lambda\le l$ and $B_\lambda(x)\subset B_l(y)$. Here $B_\lambda(x)$ stands for the closed $\lambda$-ball centered at $x\in X$.

The tower $T^L_X$ will be called the {\em canonical $L$-tower} of a metric space $X$.
Observe that for each point $x\in X$ the set $B_L(x)=\{(B_\lambda(x),\lambda):\lambda\in L\}$ is a branch of the tower $T^L_X$, so the map
$$B_L:X\to\partial T^L_X,\;\;B_L:x\mapsto B_L(x),$$ called the {\em canonical map}, is well-defined.

The following important fact was proved in Propositions 4.4 and 4.5 of \cite{BZ}.

\begin{lemma}\label{c5} Let $L\subset[0,\infty)$ be a level set. Then the canonical map $B_L:X\to\partial T_X^L$ of a metric space $X$ into the boundary of its canonical $L$-tower is a coarse equivalence. If the ultrametric space $X$ is isometrically homogeneous, then the tower $T_X^L$ is homogeneous and its boundary $\partial T_X^L$ is isometrically homogeneous.
\end{lemma}

\subsection{Tower morphisms}
A map $\varphi:S\to T$ is said to be
\begin{itemize}
\item {\em monotone} if for any $x,y\in S$ the inequality $x<y$ implies $\varphi(x)<\varphi(y)$;
\item {\em level-preserving} if there is an injective map $\varphi_{\Lev}:\Lev(S)\to\Lev(T)$ making the following diagram commutative:
$$
\begin{CD}
S@>{\varphi}>> T\\
@V{\lev_S}VV @ VV{\lev_T}V\\
\Lev(S)@>>{\varphi_{\Lev}}>\Lev(T).
\end{CD}
$$
\end{itemize}

For a monotone level-preserving map $\varphi:S\to T$ the induced map $\varphi_\Lev:\Lev(S)\to\Lev(T)$ is monotone and injective.

A monotone level-preserving map $\varphi:S\to T$ is called
\begin{itemize}
\item {\em a tower isomorphism} if it is bijective;
\item {\em a tower embedding} if it is injective;
\item {\em a tower immersion} if it is almost injective in the sense that for any points $x,x'\in S$ with $\varphi(x)=\varphi(x')$ we have $\lev_S(x\wedge x')\le\max\{\lev_S(x),\lev_S(x')\}+1$.
\end{itemize}

Each monotone map $\varphi:S\to T$ between towers induces a multi-map $\partial\varphi:\partial S\multimap\partial T$ assigning to a branch $\beta\subset S$ the set $\partial\varphi(\beta)\subset\partial T$ of all branches of $T$ that contain the linearly ordered subset $\varphi(\beta)$ of $T$. It follows that $\partial\varphi(\beta)\ne\emptyset$ and hence $(\partial\varphi)^{-1}(\partial T)=\partial S$.

We are interested in immersions of towers because the have the following their property proved in  Proposition 5.4 of \cite{BZ}.

\begin{lemma}\label{p8} Any surjective tower immersion $\varphi:S\to T$ induces a coarse equivalence $\partial \varphi:\partial S\multimap\partial T$.
\end{lemma}

Identity embeddings of level subtowers also induce coarse equivalences.

\begin{lemma}\label{c7} Let $T$ be a pruned tower and $L$ be a cofinal subset of \ $\Lev(T)$. The multi-map $\partial\id:\partial T^L\multimap \partial T$ induced by the identity embedding $\id:T^L\to T$ is a coarse equivalence.
\end{lemma}

The following lemma was proved in \cite[5.8]{BZ}.

\begin{lemma}\label{p11} Let $S,T$ be pruned towers and
$f:\Lev(S)\to \Lev(T)$ be a  monotone map. If $\Deg_\lambda^{\lambda+1}(S)\le\deg_{f(\lambda)}^{f(\lambda+1)}(T)$ for each non-maximal level $\lambda\in\Lev(S)$, then there is a tower embedding $\varphi:S\to T$ such that $\varphi_\lev=f$.
The tower embedding $\varphi$ induces a coarse embedding $\partial\varphi:\partial S\multimap\partial T$.
\end{lemma}

Our last lemma will play a key role in the proof of Theorem~\ref{main}.
It is an infinite version of (a much more technically difficult) Lemma~6.1 from \cite{BZ}.

\begin{lemma}\label{MainLemma} Let $T,S$ be two pruned ${\da}$-bounded ${\upa}$-unbounded towers and $f:\Lev(T)\to\Lev(S)$ be a monotone bijective map. Assume that for each level $\lambda\in \Lev(T)$ we get
  $$\w\le\Deg_\lambda^{\lambda+1}(T)\le\deg_{f(\lambda)}^{f(\lambda+1)}(S)\le \Deg_{f(\lambda)}^{f(\lambda+1)}(S)\le \deg_{\lambda+1}^{\lambda+2}(T).$$
Then there is a surjective tower immersion $\varphi:T\to H$, inducing the coarse equivalence $\partial \varphi:\partial T\Ra\partial H$.
\end{lemma}

\begin{proof} First we introduce some more notation. Since the towers $T,S$ are ${\da}$-bounded and ${\upa}$-unbounded, their level sets $\Lev(T)$ and $\Lev(S)$ are order isomorphic to $\w$ and will therefore be identified with $\w$. In this case $f:\Lev(T)\to\Lev(S)$ coincides with the identity map of $\w$.

A subset $A$ of the tower $T$ will be called a {\em trapezium} if $A={\da}P$ for some non-empty subset $P\subset\suc(v)$ of parents of some point $v\in T$, called the {\em vertex} of the trapezium $A$ and denoted by $\vx(A)$. It is easy to see that $\{\vx(A)\}\cup{\da}P$ is a subtower of $T$. The set $P$ generating the trapezium $A={\da}P$ will be called the {\em plateau} of the trapezium.

A map $\varphi:{\da}P\to S$ from a trapezium ${\da}P\subset T$ to the tower $S$ will be called an {\em admissible immersion} if
\begin{itemize}
\item $\varphi=\phi|{\da}P$ for some tower immersion $\phi:\{\vx({\da}P)\}\cup{\da}P\to S$,
\item there is a vertex $s\in S$ such that $\varphi(P)=\{s\}$ and $\varphi({\da}P)={\da}s$.
\end{itemize}

Lemma~\ref{MainLemma} will be derived from the following

\begin{claim}\label{cl1} For any $k\in\w$, trapezium ${\da}A_k\subset T$, and vertex $w\in S$ with $\lev(w)=\lev(A_k)=k$ and $|A_k|=\deg(w)$ there is an admissible immersion $\varphi:{\downarrow}A_k\to {\downarrow}w$. Moreover, if $k>0$ and for some points $v\in\suc(w)$, $a\in A_k$, and set $A_{k-1}\subset\suc(a)$ with $|A_{k-1}|=\deg(v)$ and $|\suc(a)\setminus A_{k-1}|=\deg(a)$ an admissible immersion $\psi:{\da}A_{k-1}\to{\da} v$ is given, then the admissible immersion $\varphi$ can be constructed so that $\varphi|{\da} A_{k-1}=\psi$.
\end{claim}

\begin{proof} This claim will be proved by induction on $k$. If $k=0$, then
${\da}A_k=A_k$ and the constant map $\varphi:A_k\to\{w\}\subset H$ is the required immersion.

Assume that the claim has been proved for some $k-1\in\w$. Fix a trapezium ${\da}A_k\subset S$ and a point $w\in T$ with $\lev_S(A_k)=\lev_T(w)=k$ and $\deg(w)=|A_k|$.
Observe that the set $\suc(A_k)=\bigcup_{a\in A_k}\suc(a)$ has cardinality
$$\deg(w)=|A_k|\le|\suc(A_k)|=\sum_{a\in A_k}\deg(a)\le|A_k|\cdot\Deg^k_{k-1}(T)\le$$ $$|A_k|\cdot \deg^k_{k-1}(S)\le|A_k|\cdot\deg(w)=\deg(w).$$
Observe also that $\deg(u)\le\deg(a)$ for every $u\in\suc(w)$ and $a\in A_k$.
This follows from $\deg(u)=0=\deg(a)$ if $k=1$ and
$$\deg(u)\le\Deg^{k-1}_{k-2}(S)\le \deg^{k}_{k-1}(T)\le \deg(a)$$ if $k>1$.

Consequently, we can write the set $\suc(A_k)$ as a disjoint union $\suc(A_k)=\bigcup_{u\in\suc(w)}A_{k-1}(u)$ of subsets of cardinality $|A_{k-1}(u)|=\deg(u)$ for $u\in\suc(w)$ such that the cover $\{A_{k-1}(u):u\in\suc(w)\}$ refines the cover $\{\suc(a):a\in A_k(u)\}$ of $\suc(A_k)$.

 By the inductive assumption, for each vertex $u\in\suc(w)$ we can find an admissible immersion $\varphi_{u}:{\downarrow} A_{k-1}(u)\to{\downarrow} u$. Now define the admissible immersion $\varphi:{\downarrow}A_k\to{\downarrow}w$ by letting
$$\varphi(x)=
\begin{cases}
\varphi_{u}(x)&\mbox{ if $x\in{\downarrow}A_{k-1}(u)$ for some $u\in\suc(w)$};\\
w&\mbox{ if $x\in A_k$}.
\end{cases}
$$

If for some vertices $a\in A_k$, $v\in \suc(w)$ and a set $A_{k-1}\subset \suc(a)$ of cardinality $|A_{k-1}|=\deg(v)$ and $|\suc(a)\setminus A_{k-1}|=\deg(a)$ an admissible immersion $\psi:{\da} A_{k-1}\to {\da} v$ is given, then we can choose the cover $\{A_{k-1}(u):u\in\suc(w)\}$ so that $A_{k-1}(v)=A_{k-1}$ and then take $\varphi_v=\psi$. In this case $\varphi|{\da} A_{k-1}=\psi$. This completes the proof of Claim~\ref{cl1}.
\end{proof}

Now we are able to complete the proof of Lemma~\ref{MainLemma}. In the towers $T$ and $S$ choose two branches $\{a_k\}_{k\in\w}\in\partial T$ and $\{b_k\}_{k\in\w}\in\partial S$ such that $\lev_T(a_k)=k=\lev_S(b_k)$ for all $l\in\w=\Lev(T)=\Lev(S)$.
For every $k\in\w$ choose a subset $A_k\subset\suc(a_{k+1})$ such that $a_k\in A_k$ and $|\suc(a_{k+1})\setminus A_k|=\deg(a_{k+1})$ and
$|A_k|=\deg(b_{k})$. Such a choice of $A_k$ is always possible because
$\deg(a_{k+1})\ge\deg^{k+1}_k(T)\ge \Deg^{k}_{k-1}(S)\ge \deg(b_{k})$.

Using Claim~\ref{cl1} inductively we can construct a sequence of admissible immersions $\varphi_k:{\da}A_k\to{\da}b_k$, $k\in\w$, such that $\varphi_{k+1}|{\da}A_{k}=\varphi_k$.
Finally, define a surjective tower immersion $\varphi:T\to S$ by letting $\varphi|{\da}A_k=\varphi_k$ for $k\in\w$.
By Lemma~\ref{p8}, the induced multi-map $\partial \varphi:\partial T\multimap\partial S$ is a coarse equivalence.
\end{proof}

\section{Proof of Theorem~\ref{main}}\label{s5}

Let $X,Y$ be two ultrametric spaces.

1. First assume that $\cov^\sharp(X)\le\cov^\flat(Y)$. In this case we shall prove that $X$ is coarsely equivalent to a subspace of $Y$.

By definition of the cardinal $\cov^\sharp(X)$, there is a number $\e_0\in\IR_+$ such that $$\cov^\sharp(X)=\sup_{\delta\in\IR_+}\big(\sup_{x\in X}\cov_{\e_0}^\delta(x)\big)^+.$$ Choose any unbounded strictly increasing number sequence $(\e_n)_{n=1}^\infty$ such that $\e_1>\e_0$, put $E=\{\e_n:n\in\w\}$ and consider the canonical tower $T^E_X=\{(B_{\e_n}(x),\e_n):n\in\w\}$ of the ultrametric space $X$. By Lemma~\ref{c5}, the canonical map $B_E:X\to\partial T_X^E$ is a coarse equivalence.

Observe that for every  $n\in\w$ the cardinal
$$\kappa_n=\Deg_{\e_n}^{\e_{n+1}}(T^E_X)=\sup_{x\in X}\cov_{\e_n}^{\e_{n+1}}(x)$$ is strictly smaller  than $\cov^\sharp(X)$.

Let $\delta_0=0$ and choose by induction on $n\in\w$  a real number $\delta_{n+1}>1+\delta_n$ such that
$$\min_{y\in Y}\cov_{\delta_n}^{\delta_{n+1}}(y)\ge \kappa_n.$$
The choice of $\delta_{n+1}$ is possible since $\cov^\flat(Y)\ge\cov^\sharp(X)>\kappa_n$.
Let $D=\{\delta_n:n\in\w\}$ and consider the canonical tower $T_Y^D$ of the ultrametric space $Y$.
By Lemma~\ref{c5}, the canonical map $B_D:Y\to \partial T_Y^D$ is a coarse equivalence.

 The choice of the sequence $(\delta_n)_{n\in\w}$ guarantees that for every $n\in\w$ we get
$$\Deg_{\e_n}^{\e_{n+1}}(T^E_X)=\sup_{x\in X}\cov_{\e_n}^{\e_{n+1}}(y)=\kappa_n\le \min_{y\in Y}\cov_{\delta_n}^{\delta_{n+1}}(y)=\deg_{\delta_n}^{\delta_{n+1}}(T^D_Y).$$
By Lemma~\ref{p11}, there is a tower embedding $\varphi:T^E_X\to T^D_Y$, which induces a coarse embedding $\partial\varphi:\partial T^E_X\multimap\partial T_Y^D$.

Then $(B_D)^{-1}\circ\partial\varphi\circ B_E:X\multimap Y$ is the required coarse embedding of $X$ into $Y$.
\smallskip

2. Now assuming that $\cov^\flat(X)=\cov^\sharp(X)=\cov^\flat(Y)=\cov^\sharp(Y)$, we shall prove that  the ultrametric spaces $X,Y$ are coarsely equivalent. We shall consider four cases,  depending on the value of the cardinal $\kappa=\cov^\flat(X)=\cov^\sharp(X)=\cov^\flat(Y)=\cov^\sharp(Y)$.

If $\kappa=0$, then the spaces $X,Y$ are empty and hence coarsely equivalent.

If $\kappa=1$, then the metric spaces $X,Y$ are bounded and hence coarsely equivalent (to a singleton).

If $\kappa=\w$, then the ultrametric spaces $X,Y$ are coarsely equivalent by Theorem~5 of \cite{BZ}.

It remains to consider the case of uncountable cardinal $\kappa$. By definition of the cardinals $\cov^\sharp(X)=\cov^\sharp(Y)$, there are numbers $\e_0,\delta_0\in\IR_+$ such that for every $\e,\delta\in\IR_+$ we get $\sup_{x\in X}\cov_{\e_0}^\e(x)<\cov^\sharp(X)=\kappa$ and $\sup_{y\in Y}\cov_{\delta_0}^\delta(y)<\cov^\sharp(Y)=\kappa$.

Using definition of the cardinals $\cov^\flat(X)=\kappa=\cov^\flat(Y)$, we can inductively  construct unbounded strictly increasing sequences $(\e_n)_{n=1}^\infty$ and $(\delta_n)_{n\in\w}$ such that
$$\sup_{x\in X}\cov_{\e_n}^{\e_{n+1}}(x)\le \min_{y\in Y}\cov_{\delta_n}^{\delta_{n+1}}(y)$$ and
$$\sup_{y\in Y}\cov_{\delta_n}^{\delta_{n+1}}(x)\le \min_{x\in X}\cov_{\e_{n+1}}^{\e_{n+2}}(y)$$
for every $n\in\w$.

Let $E=\{e_n:n\in\w\}$, $D=\{\delta_n:n\in\w\}$ and consider the canonical towers $T^E_X$ and $T^D_Y$ of the ultrametric spaces $X$ and $Y$, respectively. By Lemma~\ref{c5}, the canonical maps $B_E:X\to\partial T^E_X$ and $B_D:Y\to\partial T_Y^D$ are coarse equivalences.

Observe that for every $n\in\w$ we get
$$
\begin{aligned}
\Deg_{\e_n}^{\e_{n+1}}(T^E_X)&=\sup_{x\in X}\cov_{\e_n}^{\e_{n+1}}(x)
 \le\min_{y\in Y}\cov_{\delta_n}^{\delta_{n+1}}(y)=\deg_{\delta_n}^{\delta_{n+1}}(T^D_Y)\le\\
 &\le \Deg_{\delta_n}^{\delta_{n+1}}(T^D_Y)=\sup_{y\in Y}\cov_{\delta_n}^{\delta_{n+1}}(y)
 \le \\
&  \le \min_{x\in X}\cov_{\e_{n+1}}^{\e_{n+2}}(x)=\deg_{\e_{n+1}}^{\e_{n+2}}(T^E_X).
\end{aligned}
 $$
So, we can apply Lemma~\ref{MainLemma} in order to construct a surjective tower immersion $\varphi:T^E_X\to T^D_Y$, which induces a coarse equivalence $\partial \varphi:\partial T^E_X\multimap \partial T^D_Y$.

Then the composition $B_D^{-1}\circ\partial\varphi\circ B_E:X\multimap Y$
$$\xymatrix{X\ar^{B_E}[r]&\partial T_X^E\ar@{-}^{\partial\varphi}[r]&\kern-10pt\multimap\partial T_Y^D\ar@{-}^>>>{B_D^{-1}}[r]&\kern-5pt \multimap Y}$$
is the required coarse equivalence between $X$ and $Y$.\hfill$\square$\\

{\bf Acknowledgements.}
This research was supported in part by the Slovenian Research Agency grant No.~P1-0292.

\end{document}